\documentclass[reqno,a4paper,12pt]{amsart}
\usepackage{amsmath,amssymb,amsthm,amsfonts,amsbsy,bm}
\usepackage[mathscr]{eucal}

\numberwithin{equation}{section}

\newcommand{\G}{\Gamma}
\renewcommand{\d}{\delta}

\newcommand{\m}{\mu}

\renewcommand{\SS}{\Sigma}
\renewcommand{\t}{\tau}

\newcommand{\C}{{\mathbb C}}
\newcommand{\R}{{\mathbb R}}
\newcommand{\Z}{{\mathbb Z}}

\newcommand{\nb}{{\mathbf n}}

\newcommand{\QF}{\mathfrak Q}

\newcommand{\Fc}{{\mathcal F}}

\newcommand{\Hc}{{\mathcal H}}

\newcommand{\Lcc}{{\mathcal L}}

\newcommand{\supp}{\hbox{{\rm supp}}\,}

\newcommand{\Q}{\QF}

\newcommand{\pd}{\partial} 
\newcommand{\pa}{\bar{\partial}}

\DeclareMathOperator{\rank}{rank}

\newcommand{\CapB}{\operatorname{Cap\,}}

\newcommand{\Ker}{\hbox{{\rm Ker}}\,}
\newcommand{\Ran}{\hbox{{\rm Ran}}\,}

\newcommand{\Dom}{\operatorname{Dom\,}}

\newtheorem{theorem}{Theorem}[section]
\newtheorem{proposition}[theorem]{Proposition}

\theoremstyle{definition}

\theoremstyle{remark}
\newtheorem*{remark*}{\bf Remark}

\newcommand{\abs}[1]{\lvert#1\rvert}
\newcommand{\norm}[1]{\lVert#1\rVert}

\newcommand{\x}{X}

\begin{document}

\sloppy

\title[Eigenvalue Clusters]{Eigenvalue clusters of the Landau Hamiltonian in the exterior of a compact domain}

 \author[Pushnitski]{Alexander  Pushnitski}
\address[A. Pushnitski]{Department of Mathematics
King's College London Strand, London WC2R  2LS UK}
 \email{alexander.pushnitski@kcl.ac.uk}
\author[Rozenblum]{Grigori Rozenblum}
\address[G. Rozenblum]{Department of Mathematical Sciences \\
                        Chalmers University of Technology \\
                        and University of Gothenburg \\
                        Chalmers Tv\"argatan, 3, S-412 96
                         Gothenburg
                        Sweden}
\email{grigori@math.chalmers.se}
\begin{abstract}
We consider the Schr\"odinger operator with a constant magnetic field in the exterior 
of a compact domain on the plane. 
The spectrum of this operator consists of clusters of eigenvalues 
around the Landau levels. We discuss the rate of accumulation of eigenvalues in a fixed cluster. 
\end{abstract}
\date{19 June 2007}

\maketitle

\section{Introduction}\label{intro}

\subsection{Preliminaries}
 The  Landau Hamiltonian describes a charged
particle confined to a plane  in a constant magnetic
field. The Landau Hamiltonian is one of the earliest explicitly solvable
quantum mechanical models.
Its spectrum consists of
the Landau levels,\footnote{It is a little known fact that this was worked 
out by Fock two years before Landau; see \cite{Fock1, Landau}.}  
infinitely degenerate eigenvalues
placed at the points of an arithmetic progression.

In \cite{Uzy}, the Landau Hamiltonian was considered in the exterior of a compact obstacle.
Introducing the obstacle 
produces  clusters of eigenvalues of finite multiplicity around the Landau levels.
Various asymptotics (high energy, semiclassical) of these eigenvalue clusters were 
studied in \cite{Uzy}. 
In this paper we focus on a different aspect of 
the spectral analysis of this model: 
for a fixed eigenvalue cluster, we consider the rate of accumulation of 
eigenvalues in this cluster to the Landau level. 
We describe this rate of accumulation 
rather precisely in terms of the logarithmic
capacity of the obstacle. 

Our construction is motivated by the recent progress in the study of the Landau 
Hamiltonian on the whole plane perturbed 
by a compactly supported or fast decaying
electric or magnetic field, see \cite{RaiWar,MelRoz,FilPush,RozTa}.
In particular, we use some  operator theoretic constructions
from \cite{RaiWar} and \cite{MelRoz} and some 
concrete analysis (related to logarithmic capacity) from \cite{FilPush}.

\subsection{The Landau Hamiltonian} 
We will write $x=(x_1,x_2)\in \R^2$ and identify $\R^2$ with $\C$ 
in the standard way, setting $z=x_1+ix_2\in\C$. 
The Lebesgue measure in $\R^2$  will be  denoted by $dx$ and 
in $\C$ by $dm(z)$. 
The derivatives with respect to $x_1$, $x_2$ are denoted by $\pd_k=\pd_{x_k}$;  
we set, as usual,  $\pa=(\pd_1+i\pd_2)/2,\; \pd=(\pd_1-i\pd_2)/2$.

We denote by $B>0$ the magnitude of the constant magnetic field
in $\R^2$. We choose the gauge 
$A(x)=(A_1(x), A_2(x))=(-\frac12 Bx_2, \frac12 Bx_1)$ 
for the magnetic vector potential associated with this field. 
The magnetic Hamiltonian on the whole plane is defined as 
\begin{equation}\label{2:Schr.0}
{\x_0}=-(\nabla-i A)^2 \quad\text{ in } L^2(\R^2).
\end{equation}
More precisely, for $u\in C_0^\infty(\R^2)$ we set
\begin{equation}\label{1:Form1}
 \norm{u}_{H^1_A}^2=\int_{\R^2}\left|i\nabla u(x)+ A(x) u(x)\right|^2 dx
\end{equation}
and define $X_0$ as the selfadjoint operator which corresponds to the 
closure of the quadratic form $\norm{u}_{H^1_A}^2$, $u\in C_0^\infty(\R^2)$.

It is well known (see \cite{Fock1,Lan} or \cite{LanLif}) that the spectrum of ${\x_0}$ consists of 
the eigenvalues $\Lambda_q=(2q+1)B$, $q=0,1,\dots$, of infinite multiplicity. 
In particular, we have 
\begin{equation}
\norm{u}_{H^1_A}^2\geq B\norm{u}_{L^2}^2, 
\quad 
u\in C_0^\infty(\R^2).
\label{lowerbd}
\end{equation}
We will denote by $\Lcc_q$ the eigenspace of ${\x_0}$ corresponding to $\Lambda_q$ 
and by $P_q$ the operator of orthogonal projection onto $\Lcc_q$ in $L^2(\R^2)$. 
Later on, we will need an explicit description of $\Lcc_q$; this will be discussed 
in section~\ref{subspaces}.

Let $\Omega\subset\R^2$ be an open set. In order to define the magnetic 
Hamiltonian in $\Omega$, it is convenient to consider the associated quadratic form. 
Following  \cite{LiebLoss}, we denote by $H^1_A(\Omega)$ the closure of 
$C_0^\infty(\Omega)$ with respect to the norm $\norm{u}_{H^1_A}$. 
The quadratic form $\norm{u}_{H^1_A}^2$ is closed in $L^2(\Omega)$
and (by \eqref{lowerbd}) positively defined.  
This form defines a self-adjoint operator in $\R^2$ which we denote by ${\x}(\Omega)$. 
If $\Omega$ is bounded by a smooth curve, then the usual
computations show that this definition of
${\x}(\Omega)$ corresponds to setting the Dirichlet
boundary condition on $\partial \Omega$. The operator
${\x_0}$ corresponds to taking $\Omega=\R^2$ in the above
definitions.

\subsection{Main results}
Let $K\subset\R^2$ be a compact set and $K^c$ its complement. 
Our main results concern the spectrum of the operator ${\x}(K^c)$.  
First we state a preliminary result which
gives a general description of the spectrum of
${\x}(K^c)$. This result is
already known (see \cite{Uzy}) but as  part
of our construction, we provide a simple proof in
Section~\ref{outline}.

\begin{proposition}\label{th1}
Let $K\subset \mathbb{R}^2$ be a compact set. Then
$$
\sigma_{ess}({\x}(K^c))=\sigma_{ess}({\x_0})=\cup_{q=0}^\infty\{\Lambda_q\},
\quad \Lambda_q=(2q+1)B.
$$
Moreover, for all $q$ and all $\lambda\in(\Lambda_{q-1},\Lambda_q)$ 
the number of eigenvalues of
${\x}(K^c)$ in $(\lambda,\Lambda_q)$ is finite.
\end{proposition}
In other words, the last statement means that the eigenvalues of ${\x}(K^c)$ 
can accumulate to the Landau levels only from above.

For all $q\geq0$, we enumerate the eigenvalues of
${\x}(K^c)$ in $(\Lambda_q, \Lambda_{q+1})$:
$$
\lambda_{1}^q\geq \lambda_{2}^q \geq\dots
$$
Proposition~\ref{th1} ensures that
$\lambda_{n}^{q}\to \Lambda_q$ as
$n\to\infty$.   Below we describe the rate of this 
convergence. Roughly speaking, we will see that
for large $n$,
\begin{equation}
\frac{a^n}{n!}\leq
\lambda_{n}^{q}-\Lambda_q\leq
\frac{b^n}{n!} \label{lambdas}
\end{equation}
with some  $a,b$ depending on $K$.
In order to discuss the dependence of  $a,b$ on the
domain $K$, let us introduce the following notation:
\begin{equation}
\begin{split}
\Delta_q(K)&=\limsup_{n\to\infty}
[n!(\lambda_{n}^q-\Lambda_q)]^{1/n},
\\
\delta_q(K)&=\liminf_{n\to\infty}
[n!(\lambda_{n}^{q}-\Lambda_q)]^{1/n}.
\end{split}
\label{deltas}
\end{equation}
The estimates for these spectral characteristics will
be given in terms of the logarithmic capacity of $K$ 
which is denoted by $\CapB(K)$.  
For the definition and properties of logarithmic capacity, we refer to 
\cite{Lan}.  
We will also need a version of inner capacity, which we
denote by   $\CapB_-(K)$ and define by
$$
\sup\{\CapB S\mid \text{ $S\subset K$ is a domain with a Lipschitz boundary}\}.
$$
By $Pc(K)$ we denote the polynomial convex hull of $K$.
$Pc(K)$ can be alternatively described as the complement
of the unbounded connected component of $K^c$.
It is well known that $\CapB(K)=\CapB(Pc(K))$ for any compact $K$.

\begin{theorem}\label{th2}
Let $K\subset \mathbb{R}^2$ be a compact set;
then for all $q\geq 0$ one has
\begin{align*}
\Delta_q(K)\leq \frac{B}2 (\CapB(K))^2, \quad
\\
\delta_q(K)\geq \frac{B}2 (\CapB_-(Pc(K)))^2. \quad
\end{align*}
\end{theorem}
The lower bound in the above theorem is strictly positive
if and only if  the compact $K$ has a non-empty interior. 
In particular, for such compacts the number of
eigenvalues $\lambda_1^q$, $\lambda_2^q$, \dots 
is infinite for each $q$.
However, even for some compacts without interior points,
lower spectral  bounds can be obtained.
In particular, this can be done for the
compact $K$ being a smooth (not necessarily
closed)  curve.

\begin{theorem}\label{th3}
Let $K\subset \mathbb{R}^2$ be a $C^\infty$ smooth simple curve.
Then for all $q\geq0$, one has
$$
\Delta_q(K)=
\delta_q(K)=
\frac{B}{2}(\CapB(K))^2.
$$
\end{theorem}
\begin{remark*}
\begin{enumerate}
\item
One can prove that 
\begin{equation}
\text{if $\CapB(K)=0$, then $C_0^\infty(K^c)$ is dense in $H^1_A(\R^2)$.}
\label{polar}
\end{equation}
It follows that for $K$ of zero capacity, $H^1_A(K^c)=H^1_A(\R^2)$ and therefore $\x(K^c)=\x_0$.  
Thus, for such  $K$ the spectrum of $\x(K^c)$
consists of Landau levels $\Lambda_q$.

The statement \eqref{polar} seems to be  well known to the experts in the field 
although it is difficult to pinpoint the exact reference.
One can use the argument of  \cite{AdHed}, Theorem 9.9.1; this argument applies to the usual $H^1$ 
Sobolev norm, but it is very easy to modify it for the norm $H^1_A$.  
In this theorem the Bessel capacity rather than the logarithmic capacity is used; however, 
the Bessel capacity of a compact set vanishes if and only if its logarithmic capacity vanishes. 
In order to prove the latter fact (again, well known to experts) one has to 
combine Theorem~2.2.7 in \cite{AdHed} and Sect.II.4 in \cite{Lan}.
\item
We do not know whether it  possible for $\Lambda_q$ to remain eigenvalues of
$\x(K^c)$ of infinite multiplicity if $\CapB(K)>0$.
\item
Following the proof  of Theorem~\ref{th2} and using the results of \cite{FilPush}, it is easy to show that for 
$q=0$,  the lower bound in this theorem can be replaced by the following one:
\begin{gather*}
\delta_0(K)\geq \frac{B}2 (\CapB(K_-))^2,
\\
K_-=
\{z\in\C\mid\limsup_{r\to+0}
\frac{\log m(Pc(K)\cap D_r(z))}{\log r}<\infty\}, 
\end{gather*}
where $D_r(z)=\{\zeta\in\C\mid \abs{\zeta-z}\leq r\}$, and
$m(\cdot)$ is the Lebesgue measure.
\item
Analysing the proof of Theorem~\ref{th3}, it is easy to see that if we are only interested 
its statement for finitely many $q$, it suffices to require some finite smoothness of the 
curve $K$. 
\end{enumerate}
\end{remark*}

\subsection{Outline of the proof}\label{outline}
Let us write $L^2(\R^2)=L^2(K^c)\oplus L^2(K)$.
(If the Lebesgue measure of $K$ vanishes then, of course, 
$L^2(K)=\{0\}$.) With respect to this decomposition, let us define
\begin{equation}
R(K^c)={\x}(K^c)^{-1}\oplus 0
\text{ in }
L^2(\R^2)=L^2(K^c)\oplus L^2(K).
\label{b11}
\end{equation}
Clearly, for any $\lambda\not=0$
we have
\begin{equation}
\lambda\in\sigma({\x}(K^c)) \Leftrightarrow \lambda^{-1}\in\sigma(R(K^c))
\label{invariance}
\end{equation}
with the same multiplicity. Thus, it suffices for our
purposes to study the spectrum of the operator
$R(K^c)$. 

First note that in the ``free" case $K=\emptyset$ we have 
$R(\R^2)=\x_0^{-1}$ and the spectrum of $\x_0^{-1}$ 
consists of the eigenvalues $\Lambda_q^{-1}$ of infinite
multiplicity and their point of accumulation, zero. 

Next, it turns out (see section~\ref{ReductionToToeplitz}) that 
\begin{equation}
R(K^c)=\x_0^{-1}-W, 
\quad\text{where $W\geq0$ is compact.}
\label{b10}
\end{equation}
Thus, the Weyl's theorem on the invariance of the essential spectrum under compact perturbations
ensures that $\sigma_{ess}(R(K^c))=\sigma_{ess}(\x_0^{-1})$. 
Moreover, a simple operator theoretic argument (see e.g. \cite[Theorem~9.4.7]{BS})
shows that the eigenvalues
of $R(K^c)$ do not accumulate to the inverse Landau levels $\Lambda_q^{-1}$ 
from above. Thus, the spectrum of $R(K^c)$ consists of zero and the eigenvalue 
clusters $\{(\lambda_1^q)^{-1},(\lambda_2^q)^{-1},\dots \}$  
with the eigenvalues in the $q$'th cluster accumulating to $\Lambda_q^{-1}$. 
In section~\ref{accumulation} we show that the rate of accumulation of
$(\lambda_n^q)^{-1}$ to $\Lambda_q^{-1}$ can be described in terms of the spectral asymptotics of the 
Toeplitz type operator $P_q W P_q$; here $W$ is defined by \eqref{b10} 
and $P_q$ is the projection onto 
$\Lcc_q=\Ker({\x_0}-\Lambda_q)=\Ker (\x_0^{-1}-\Lambda_q^{-1})$.

The spectrum of $P_q W P_q$ is studied in sections~\ref{Toeplitz} and \ref{lastsec}, using 
the results of \cite{FilPush}.

\subsection{Acknowledgements}
A large part of this work was completed during the authors' stay at  the Isaac Newton Institute 
for Mathematical Sciences (Cambridge, UK) in the framework of the programme 
``Spectral Theory and Partial Differential Equations''. It is a pleasure to thank the Institute
and the organisers of the programme for providing this opportunity. The authors are also 
grateful to Uzy Smilansky for useful discussions.

\section{Some abstract results}\label{nonsense}
Here we collect some general operator theoretic
statements that are used in the proof. The statements
themselves, with the exception of the last one, are
 almost obvious, but spelling them out
explicitly helps explain the main ideas of our
construction.
\subsection{Quadratic forms}\label{qforms}
Our arguments    can be stated most succinctly
if we are allowed to deal with quadratic forms
whose domains are not necessarily dense in the Hilbert space.
 Here is the corresponding abstract framework;
related constructions appeared before in the literature; see e.g. \cite{Simon2}.

Let $a$ be a closed positive definite quadratic form
in a Hilbert space $\Hc$ with the domain $d[a]$. Let
the closure of $d[a]$ in $\Hc$ be $\Hc_a$. 
Then the form $a$ defines a
self-adjoint operator $A$ in $\Hc_a$.  Let
$J_a:\Hc_a\to \Hc$ be the natural embedding operator;
its adjoint $J_a^*: \Hc\to \Hc_a$ acts as the
orthogonal projection onto the subspace $\Hc_a$ of
$\Hc$. The operator  $J_a A^{-1} J_a^*$ in $\Hc$ can be considered as
the direct sum
$$
J_a A^{-1} J_a^*=A^{-1}\oplus 0
\text{ in the decomposition }
\Hc=\Hc_a\oplus \Hc_a^\perp;
$$
here we have in mind \eqref{b11}.
Now let $b$ be another closed positive definite form in $\Hc$ and let
$B$, $d[b]$, $\Hc_b$, $J_b$ be the corresponding
objects constructed for this form.
\begin{proposition}\label{th4}
Suppose that $d[b]\subset d[a]$ and $b[x,y]= a[x,y]$ for all $x,y\in d[b]$.
Then:

(i) $J_b B^{-1} J_b^*\leq J_a A^{-1} J_a^*$ on $\Hc$;

(ii) if $x\in d[b]\cap \Dom(A)$, then $x\in \Dom (B)$, $Bx=Ax$,
and $J_b B^{-1} J_b^* Ax=J_a A^{-1}J_a^*Ax$.
\end{proposition}
\begin{proof}
It suffices to consider the case $\Hc_a=\Hc$.

(i)
The hypothesis implies
$$
b[x,x]=a[J_b x,J_b x]
\quad \text{ for all } x\in d[b].
$$
This can be recast as 
$\norm{B^{1/2}x}=\norm{A^{1/2}J_b x}$, $x\in d[b]$. 
It follows that the operator $A^{1/2} J_b B^{-1/2}$ is an isometry  on
$\Hc_b$ and therefore $A^{1/2} J_b B^{-1/2}J_b^*$
is a contraction on $\Hc$. 
By conjugation, we get that $\norm{J_b B^{-1/2}J_b^* A^{1/2} z}\leq \norm{z}$ 
for all $z\in d[a]$. The last statement is equivalent to $\norm{J_b B^{-1/2}J_b^*  u}\leq \norm{A^{-1/2}u}$ 
for all $u\in \Hc$, and so $ J_b B^{-1} J_b^*\leq A^{-1}$ as
required.

(ii) Let $y\in d[b]$; then
$$
b[x,y]=a[x,y]=(Ax,y),
$$
and so $x\in\Dom(B)$ and $Bx=Ax$. Next,
$J_a A^{-1}J_a^*Ax=A^{-1}Ax=x$, and
$$
J_b B^{-1}J_b^*Ax=J_b B^{-1}J_b^*Bx=J_bB^{-1}Bx=J_bx=x,
$$
which proves the required statement.
\end{proof}

\subsection{Shift in enumeration}\label{shift}
 The
asymptotics of the type discussed in
Theorems~\ref{th2} and \ref{th3} is independent of a
shift in the enumeration of eigenvalues. This is a
consequence of the following elementary fact. Let
$b_1\geq b_2\geq\dots$ be a sequence of positive
numbers such that $\limsup_{n \to\infty}[n!
b_n]^{1/n}<\infty$. Then for all $\ell\in\Z$,
\begin{equation}
\lim_{n\to\infty}\genfrac{\{}{\}}{0pt}{}{\sup}{\inf}[n!b_{n+\ell}]^{1/n}
=
\lim_{n\to\infty}\genfrac{\{}{\}}{0pt}{}{\sup}{\inf}[n!b_n]^{1/n}.
\label{a1}
\end{equation}

\subsection{Accumulation of eigenvalues}\label{accumulation}
Having in mind \eqref{b10}, let us consider the following general 
situation. 
Let $T$ be a self-adjoint operator and let $\Lambda$
be an isolated eigenvalue of $T$ of infinite
multiplicity with the corresponding eigenprojection
$P_\Lambda$. Let $\t>0$ be such that
$$
((\Lambda-2\t,\Lambda+2\t)\setminus\{\Lambda\})\cap
\sigma(T)=\emptyset.
$$
Next, let $W\geq0$ be a compact
operator; consider the spectrum of $T-W$.  
The Weyl's theorem on the invariance of
the essential spectrum under compact perturbations ensures that
$$
((\Lambda-2\t,\Lambda+2\t)\setminus\{\Lambda\})\cap
\sigma_{ess}(T-W)=\emptyset.
$$
Moreover, a simple argument  (see e.g.
\cite[Theorem~9.4.7]{BS}) shows that the eigenvalues
of $T-W$ do not accumulate to $\Lambda$ from above (i.e.
$(\Lambda,\Lambda+\epsilon)\cap \sigma(T-W)=\emptyset$
for some $\epsilon>0$).

We will need a description of the eigenvalues of $T-W$
below $\Lambda$ in terms of the eigenvalues of the
Toeplitz operator $P_\Lambda W P_\Lambda$. Let
$\mu_1\geq\mu_2\geq\cdots$ be the eigenvalues of
$P_\Lambda  W P_\Lambda$; in order to exclude
degenerate cases, let us assume that this operator has
infinite rank. Let $\lambda_1\leq \lambda_2\leq\cdots$
be the eigenvalues of $T-W$ in the interval
$(\Lambda-\tau,\Lambda)$.

\begin{proposition}\label{th5}
Under the above assumptions, for any $\epsilon>0$ there exists $\ell\in\Z$
such that for all sufficiently large $n$, one has
$$
(1-\epsilon)\mu_{n+\ell}\leq\Lambda-\lambda_n\leq (1+\epsilon)\mu_{n-\ell}.
$$
\end{proposition}
The proof borrows its key element from \cite[Lemma~1.1]{IwaTam}. 
An alternative proof can be found in   \cite[Proposition~4.1]{RaiWar}.

\begin{proof}

1.
We denote $S=T-W$ and $Q_\Lambda=I-P_\Lambda$
and  consider the operators
$$
R_\pm
=
\epsilon P_\Lambda W P_\Lambda +\frac1{\epsilon} Q_\Lambda W Q_\Lambda
\pm
(P_\Lambda W Q_\Lambda+Q_\Lambda W P_\Lambda).
$$
and
$$
S_\pm = P_\Lambda(T-(1\pm\epsilon)W)P_\Lambda 
+
Q_\Lambda(T-(1\pm\frac1\epsilon)W)Q_\Lambda.
$$
We have 
$$
S=S_++R_-=S_--R_+.
$$

2. 
Since $W$ is compact, the operators $R_\pm$ are also compact.
Since  $R_\pm$ can be represented as
$$
R_\pm=(\sqrt{\epsilon} P_\Lambda\pm \frac1{\sqrt{\epsilon}} Q_\Lambda)
W(\sqrt{\epsilon} P_\Lambda\pm \frac1{\sqrt{\epsilon}} Q_\Lambda)
$$
and $W\geq0$, we see that $R_\pm\geq0$.

3. Let us discuss the spectrum of $S_\pm$ in
$(\Lambda-\tau, \Lambda)$. Clearly, the spectrum of
$P_\Lambda(T-(1\pm\epsilon)W)P_\Lambda
=\Lambda P_\Lambda-(1\pm\epsilon)P_\Lambda WP_\Lambda$ 
consists of the eigenvalues $\Lambda-(1\pm\epsilon)\mu_n$. Next,
since by assumption, $T\mid_{\Ran Q_\Lambda}$ has no
spectrum in $(\Lambda-2\t,\Lambda+2\t)$ and $W$
is compact, we see that
$Q_\Lambda(T-(1\pm\frac1\epsilon)W)Q_\Lambda
\mid_{\Ran Q_\Lambda}$ has only finitely many eigenvalues in the
interval $(\Lambda-\tau,\Lambda+\tau)$. Since the operators
$P_\Lambda(T-(1\pm\epsilon)W)P_\Lambda$ and
$Q_\Lambda(T-(1\pm\frac1\epsilon)W)Q_\Lambda$ act in
orthogonal subspaces of our Hilbert space, the
spectrum of $S_\pm$ is the union of the spectra of
these operators.

So we arrive at the following conclusion. Let
$\nu^\pm_1\leq \nu^\pm_2\leq\cdots$ denote the
eigenvalues of $S_\pm$ in $(\Lambda-\tau, \Lambda)$.
Then
\begin{equation}
\nu^+_n=\Lambda-(1+\epsilon)\mu_{n-i},
\quad
\nu^-_n=\Lambda-(1-\epsilon)\mu_{n-j},
\label{b3}
\end{equation}
for some integers $i,j$ and all
sufficiently large $n$.

4. Let us prove that $\lambda_n\leq \nu_{n+k}^-$ for
some integer $k$\ and all sufficiently large $n$. 
Denote $\delta=(\lambda_1-\Lambda+\tau)/2$ and let us write 
$R_+=R_+^{(1)}+R_+^{(2)}$, where $0\leq R_+^{(1)}\leq \delta I$
and  $\rank R^{(2)}_-<\infty$. 
Denote by $N_S(\alpha,\beta)$ the number
of eigenvalues of $S$ in the interval $(\alpha,\beta)$. 
Writing $S=S_- -R_+^{(1)}-R_+^{(2)}$,  we get for any 
$\lambda\in(\lambda_1,\Lambda)$:
\begin{multline*}
N_S(\Lambda-\tau,\lambda)
=
N_S(\lambda_1-2\delta,\lambda)
\\
\geq
N_{S_--R_+^{(1)}}(\lambda_1-2\delta,\lambda)
-\rank R^{(2)}_- 
\geq 
N_{S_-}(\lambda_1-\delta, \lambda)-\rank R^{(2)}_-.
\end{multline*}
The second inequality above follows from $\sigma(R_+^{(1)})\subset[0,\delta]$
(see \cite[Lemma~9.4.3]{BS}).
These inequalities for the eigenvalue counting functions can be rewritten
as $\lambda_n\leq \nu_{n+k}^-$ with some integer $k$. 

In the same way, one proves that $\lambda_n\geq \nu_{n-k}^+$ for large $n$ and some
integer $k$. Taken together with \eqref{b3}, this yields the required result.
\end{proof}

\section{Preliminaries and reduction to Toeplitz operators}\label{ReductionToToeplitz}
Let $K\subset \R^2$ be a compact set; 
we return to the discussion of the spectrum of  ${\x}(K^c)$ and start
with some general remarks.

First we  would like to point out that the spectral asymptotics that 
we are interested in is independent of the ``holes" in
the domain $K$:
\begin{equation}
\delta_q(K)=\delta_q(Pc(K)), 
\quad
\Delta_q(K)=\Delta_q(Pc(K)).
\label{holes}
\end{equation}
Indeed, let us write $K^c=\Omega\cup \SS$, 
where $\Omega$ is the unbounded
connected component of $K^c$ and $\Omega$ and $\SS$
are disjoint. With respect to the direct sum decomposition 
$L^2(K^c)=L^2(\Omega)\oplus L^2(\SS)$, we have
${\x}(K^c)={\x}(\Omega)\oplus {\x}(\SS)$. By the compactness of
the embedding $H^{1}_A(\SS)\subset L^2(\SS)$, the
operator ${\x}(\SS)$ has a compact resolvent. Thus, on
any bounded interval of the real line the spectrum of ${\x}(K^c)$ differs from the
spectrum of ${\x}(\Omega)$  by at most finitely many eigenvalues.
By \eqref{a1}, this yields \eqref{holes}.

Next, we apply the abstract reasoning of section~\ref{qforms}
to the quadratic form $a[u]=\|u\|^2_{H^1_A(K^c)}$ with
domain $d[a]=H^{1}_A(K^c)$, considering $L^2(\R^2)$
as the main Hilbert space $\Hc$.
We  consider the operator $R(K^c)$ (see \eqref{b11})
and write $R(K^c)=\x_0^{-1}-W$. 
Proposition~\ref{th5} suggests that in order to
find the rate of accumulation of the eigenvalues of $R(K^c)$ 
to $\Lambda_q^{-1}$, one should
study the spectrum of the Toeplitz type operators
$P_qWP_q$. This is done in the next section. Denote by
$\mu^q_1\geq\mu^q_2\geq\dots$ the eigenvalues of
$P_qWP_q$. We will prove
\begin{proposition}\label{th6}
Let $K\subset\R^2$ be a compact set and $q\geq0$. Then
\begin{align*}
\limsup_{n\to\infty}(n! \mu^q_n)^{1/n}
&\leq
\frac{B}2 (\CapB(K))^2,
\\
\liminf_{n\to\infty}(n! \mu^q_n)^{1/n}
&\geq
\frac{B}2 (\CapB_-(K))^2.
\end{align*}
If $K$ is a $C^\infty$ smooth curve, then one has
$$
\lim_{n\to\infty}(n! \mu^q_n)^{1/n}=
\frac{B}2 (\CapB(K))^2.
$$
\end{proposition}

Now we can prove our main statements. 
\begin{proof}[Proof of Theorem~\protect\ref{th1} and Theorem~\protect\ref{th2}]
Combining  Proposition~\ref{th6}, Proposition~\ref{th5} and \eqref{holes}, 
we get the estimates for the quantities
\begin{align*}
\lim_{n\to\infty}\sup[n!(\Lambda_q^{-1}-(\lambda_n^q)^{-1})]^{1/n}
&\leq 
\frac{B}2 (\CapB(K))^2,
\\
\lim_{n\to\infty}\inf[n!(\Lambda_q^{-1}-(\lambda_n^q)^{-1})]^{1/n}
&\geq 
\frac{B}2 (\CapB_-(Pc(K))^2
\end{align*}
for any compact $K$. If $K$ is a $C^\infty$ smooth curve, we get
$$
\lim_{n\to\infty}[n!(\Lambda_q^{-1}-(\lambda_n^q)^{-1})]^{1/n}
=
\frac{B}2 (\CapB(K))^2.
$$
An elementary argument shows that
$$
\lim_{n\to\infty}\genfrac{\{}{\}}{0pt}{}{\sup}{\inf}[n!(\Lambda_q^{-1}-(\lambda_n^q)^{-1})]^{1/n}
=
\lim_{n\to\infty}\genfrac{\{}{\}}{0pt}{}{\sup}{\inf}[n!(\lambda_n^q-\Lambda_q)]^{1/n}.
$$
This yields the required statements.
\end{proof}

\begin{proof}[Proof of \eqref{b10}]

Let $D$ be a disc such that $K\subset D$. By Proposition~\ref{th4}(i), we get
$$
D^c\subset K^c \subset \R^2
\Rightarrow
R(D^c) \leq R(K^c) \leq \x_0^{-1}
$$
and so 
\begin{equation}
0\leq \x_0^{-1}-R(K^c)\leq \x_0^{-1}-R(D^c). 
\label{b12}
\end{equation}
Thus, $W=\x_0^{-1}-R(K^c)$ is non-negative; 
 let us address compactness. 
 
It is well known that if $0\leq V_1\leq V_2$ are self-adjoint operators and
$V_2$ is compact, then $V_1$ is also compact. 
Thus, by \eqref{b12}, in order to prove the compactness of $W$,
it suffices to check that $\x_0^{-1}-R(D^c)$ is compact.

Let $\G=\partial D$. 
Employing the same arguments as in the proof of \eqref{holes}, 
we see that
${\x}(\Gamma^c)^{-1}-R(D^c)$ is the inverse of the magnetic operator on the disc and
hence a  compact operator. Thus, it suffices to prove that the difference
$$
\x_0^{-1}-{\x}(\Gamma^c)^{-1}
=
(\x_0^{-1}-R(D^c))
-
({\x}(\Gamma^c)^{-1}-R(D^c))
$$ 
is compact.

Let us compute the quadratic form of this difference.
Let $f,g\in L^2(\R^2)$, $\x_0^{-1}f=u$, ${\x}(\Gamma^c)^{-1}g=v$.
We have
$$
((\x_0^{-1}-{\x}(\Gamma^c)^{-1})f,g)
=
(u,{\x}(\Gamma^c)v)-({\x_0}u,v).
$$
Integrating by parts and noting that $v\in\Dom ({\x}(\Gamma^c))$ vanishes on
$\Gamma$, we get
\begin{equation}\label{3:q-form}
(u,{\x}(\Gamma^c)v)-({\x_0} u,v) 
= 
\int_{\Gamma}(n_A v(s)^+ + n_A v(s)^-) u(s)ds
\end{equation}
where $n_A v(s)=(\nabla-iA(s))v\cdot\nb(s)$, 
$\nb(s)$ is the exterior normal to
$\G$ at the point $s$ and the superscripts
$+$ and $-$ indicate that the limits of the functions
are taken on the circle $\G$ by approaching it from
the outside or inside. 

Take a smooth cut-off function
$\omega\in C^\infty_0(\R^2)$ such that $\omega(x)=1$ in the
neighborhood of $D$. Then we can replace $u,v$ by
$u_1=\omega u$, $v_1=\omega v$ in the r.h.s. of \eqref{3:q-form}. 
By the local elliptic regularity  we
have $u_1\in H^2(\R^2)$, $v_1\in H^2( \G^c)$, 
and the corresponding Sobolev norms of $u_1$, $v_1$ 
can be estimated  via the $L^2$-norms of $f,g$. 
Now it remains to notice that
the trace mapping $u_1\mapsto u_1|_\G$ is compact as
considered from $H^2(\R^2)$ to $L^2(\G)$,
and the mappings $v_1\mapsto (n_A v_1)^\pm$ 
are compact as considered from $H^2(\G^c)$ to $L^2(\G)$. 
It follows that the difference $\x_0^{-1}-{\x}(\Gamma^c)^{-1}$ is compact, 
as required. 
\end{proof}

\section{The spectrum of Toeplitz operators}\label{Toeplitz}

\subsection{Restriction operators and the associated Toeplitz operators}\label{restriction}
Let $\mu$ be a finite measure in $\R^2$ with a compact
support. Consider the restriction operator 
$$
\gamma_0: C_0^\infty(\R^2)\ni u\mapsto u\mid_{supp(\mu)}\in
L^2(\mu).
$$
We are interested in two
special cases, namely when $\mu$ is the restriction of
the Lebesgue measure to a set with Lipschitz boundary
and when $\mu$ is the arc length measure on a simple smooth
curve. In both cases $\gamma_0$ can be extended by continuity
to a bounded and compact operator $\gamma: H^1_A(\R^2)\to L^2(\mu)$.

Next, let $J: H^1_A(\R^2)\to L^2(\R^2)$ be the embedding operator, 
$J:u\mapsto u$. Then the adjoint $J^*: L^2(\R^2)\to H^1_A(\R^2)$ acts as 
$J^*:u\mapsto {\x_0}^{-1/2}u$. 

For $q\geq0$, consider the operators $T_q(\mu)$ in $L^2(\R^2)$ defined by  the quadratic form 
$$
(T_q(\mu)u,u)_{L^2(\R^2)}
=
\int_{\supp \mu} \abs{(P_q u)(x)}^2 d\mu(x), 
\quad u\in L^2(\R^2). 
$$
This operator can be represented as 
$$
T_q(\mu)=(\gamma J^* {\x_0}^{1/2}P_q)^* (\gamma J^* {\x_0}^{1/2}P_q)
=
\Lambda_q (\gamma J^* P_q)^*(\gamma J^* P_q). 
$$
Since $\gamma$ is compact by assumption, the operator $T_q(\mu)$ is also compact.

Fix $q\geq0$; let $s^q_1\geq s^q_2\geq\dots$ be the
eigenvalues of  $T_q(\mu)$ in $L^2(\R^2)$.

\begin{proposition}\label{prop1}
(i) Let $\mu$ be the restriction of the Lebesgue measure onto
a bounded domain $K\subset \R^2$ with  a Lipschitz boundary.
Then
$$
\lim_{n\to\infty}(n! s_n^q)^{1/n}
=
\frac{B}{2} (\CapB(K))^2.
$$
(ii) Let $\mu$ be the arc length measure of a  $C^\infty$ smooth curve.
Then
$$
\lim_{n\to\infty}(n! s_n^q)^{1/n}
=
\frac{B}{2} (\CapB(\supp \mu))^2.
$$
\end{proposition}

Before proving this proposition, we need to recall the description of 
the subspaces $\Lcc_q$. 

\subsection{The structure of subspaces $\Lcc_q$}\label{subspaces}
Denote $\Psi(z)=\frac{1}{4}B|z|^2$.
Let us  define the creation and annihilation operators (first introduced
in this context by Fock \cite{Fock})
\begin{align*}
\Q=-2i e^{-\Psi}\overline{\partial}e^{\Psi}
&=-2i \overline{\partial}-\frac{B}2 iz
\\
\Q^*=-2i e^{\Psi}\partial e^{-\Psi}
&=-2i \partial+\frac{B}2 i\overline{z}.
\end{align*}
The Landau Hamiltonian  can be expressed as
\begin{equation}
{\x_0}=\Q^*\Q+B=\Q\Q^*-B. 
\label{b13}
\end{equation}
The spectrum and spectral subspaces of ${\x_0}$  can be described
in the following way.  The
equation $({\x_0}-B) u=0$ is equivalent to
\begin{equation}
\label{2:sol.0.1} \nonumber
 {\Q}u=-2i e^{-\Psi}\pa( e^{\Psi}u)=0.
\end{equation}
This means that
$f=e^{\Psi}u$ is an entire analytic function 
such that $e^{-\Psi}f\in L^2(\C)$. 
The space of such functions $f$
is called Fock or Segal-Bargmann space $\Fc^2$ 
(see \cite{Folland} for an  extensive discussion).
 So 
 $\Lcc_0=e^{-\Psi}\Fc^2$. Further 
eigenspaces  $\Lcc_q$, $q=1,2,\dots,$ are obtained as
$\Lcc_q=(\Q^*)^q\Lcc_0$. The operators $\Q^*, \Q$ act
between the subspaces $\Lcc_q$ as 
 \begin{equation}\label{1:CrAnn}
 \Q^*:\Lcc_q\mapsto \Lcc_{q+1},
\quad 
\Q :\Lcc_q\mapsto\Lcc_{q-1},
\quad
\Q:\Lcc_0\mapsto \{0\},
\end{equation}
and are, up to constant factors, isometries on $\Lcc_q$.
In particular, the substitution
\begin{equation} 
\Lcc_q\ni u=C_q^{-1}(\Q^*)^q
e^{-\Psi} f, \quad f\in \Fc^2, 
\quad C_q=\sqrt{q! (2B)^q}
\label{unitary}
\end{equation}
gives a unitary equivalence of spaces $\Lcc_q$ and $\Fc^2$.

\subsection{Proof of Proposition~\ref{prop1}}

(i) The proof is given in \cite[Lemma~3.1]{FilPush} for $q=0$ 
and \cite[Lemma~3.2]{FilPush} for $q\geq0$.

(ii)
For $q=0$ the result again follows from 
Lemma 3.1 in \cite{FilPush}. Although the reasoning
there concerns the operators 
$T_q(v)=(vP_q)^*(vP_q)$ where the function $v$ is
separated from zero on a compact, it goes through for  
$T_q(\mu)$. Only notational changes are required;  one simply has to replace the
measure $v(z)dm(z)$ by $d\m (z)$. 

For  $q\geq1$ below we apply  the
reduction to the lowest Landau level
similar to the proof of Lemma 3.2 in \cite{FilPush}.

Denote $d\tilde \mu(z)=e^{-\Psi(z)}d\mu(z)$. 
Applying  the unitary equivalence \eqref{unitary}, 
we get for $u\in \Lcc_q$
\begin{equation}\label{b9}
(T_q(\mu)u,u)_{L^2(\R^2)}
=
C_q^{-2}
\norm{(2\partial-B\overline{z})^q f}^2_{L^2(\tilde\mu)}.
\end{equation}
In particular, for $q=0$
\begin{equation}\label{b90}
(T_0(\mu)u,u)_{L^2(\R^2)}
=
C_0^{-2}
\norm{f}^2_{L^2(\tilde\mu)}.
\end{equation}
Below we separately prove the upper and lower bound for the 
quadratic form \eqref{b9}. 

1. \emph{Upper bound.}
Consider the open $\d$-neigh\-bor\-hood $U_\d\subset\C^1$ of the
curve $\G$. As it follows from the Cauchy integral
formula, for some constant $C_1(q,\delta)$, the inequality
$$
\norm{\partial^k f}_{L^2(\tilde \mu)}^2
\leq
C_1(q,\delta)
\int_{U_\delta}\abs{f(z)}^2dm(z).
$$
holds for all functions $f\in\Fc^2$. 
Thus, we have the estimate 
$$
\norm{(2\partial-B\overline{z})^q f}^2_{L^2(\tilde\mu)}
\leq 
C_2(q,\delta)
\int_{U_\delta}\abs{f(z)}^2dm(z).
$$
Using \eqref{b9}, \eqref{b90}, we arrive at the estimate
\begin{equation}
T_q(\mu)\leq C T_0(\chi_{{}_{U_\delta}}(x)dx),
\label{b4}
\end{equation}
where $\chi_{{}_{U_\delta}}$ is the characteristic function of the set $U_\delta$. 
Now we can again apply the estimate of \cite[Lemma~3.1]{FilPush} to the 
eigenvalues $s_1(\delta)\geq s_2(\delta)\geq \dots$ of $T_0(\chi_{{}_{U_\delta}}(x)dx)$.
This estimate together with \eqref{b4} yields 
$$
\lim_{n\to\infty}(n! s_n)^{1/n}
\leq
\lim_{n\to\infty}(n! s_n(\d))^{1/n}
\leq
\frac{B}{2}(\CapB (U_\d))^2.
$$
Finally, $\CapB (U_\d)\to \CapB(\G)$ as $\d\to 0$, and
this proves the upper bound.

2. \emph{Lower bound.} The lower bound for the spectrum of $T_q(\mu)$ requires a little more
work. 
Let $\sigma:[0,s]\to \C$ be the parameterization of $\Gamma$ by the arc length. 
Since $f$ is analytic, we have
\begin{equation}
(\partial f)(\sigma(t))=\rho(t)\frac{d}{dt} f(\sigma(t))
\label{b6}
\end{equation}
with some smooth factor $\rho(t)$, $\abs{\rho(t)}=1$. 

Next, due to the compactness of the embedding $H^1(0,s)\subset L^2(0,s)$, 
for any $\beta>0$ there exists a subspace of $H^1(0,s)$ of a finite 
codimension such that for any $u$ in this subspace, 
\begin{equation}
\int _0^s \abs{u(t)}^2 dt \leq \beta^2 \int_0^s \abs{u'(t)}^2 dt. 
\label{b7}
\end{equation}
It follows from \eqref{b6} and \eqref{b7} that for any $\beta>0$ 
there exists a subspace of $\Fc^2$ of a finite codimension such that for 
any $f$ in this subspace 
$$
\int_\Gamma \abs{f(z)}^2 d\tilde\mu(z)
\leq 
\beta^2 \int_\Gamma \abs{\partial f(z)}^2d\tilde\mu(z). 
$$
Arguing by induction, we obtain that for any $\beta>0$ there exists a
subspace $N=N(\beta,q)\subset\Fc^2$ of finite codimension
such that for all $f\in N(\beta,q)$
\begin{equation}
\int_\Gamma \abs{\partial^k f(z)}^2 d\tilde\mu(z)
\leq 
\beta^2 \int_\Gamma \abs{\partial^q f(z)}^2d\tilde\mu(z), 
\quad \forall k=0,1,\dots, q-1.
\label{b8}
\end{equation}
Using \eqref{b8} and choosing $\beta$ sufficiently small, we can estimate 
the form \eqref{b9} from below as follows: 
\begin{multline*}
\norm{(2\partial - B\overline z)^q f}^2_{L^2(\tilde\mu)}
\geq
(\norm{(2\partial)^q f}_{L^2(\tilde\mu)}-\sum_{k=0}^{q-1} C_{q,k}\norm{\partial^k f}_{L^2(\tilde \mu)})^2
\\
\geq
\norm{(2\partial)^q f}_{L^2(\tilde\mu)}^2(1-\sum_{k=0}^{q-1} C_{q,k}\beta)^2
=C_1 \norm{\partial^q f}_{L^2(\tilde\mu)}^2
\geq
C_2 \norm{ f}_{L^2(\tilde\mu)}^2
\end{multline*}
for all $f\in N(\beta,q)$.
Using \eqref{b9} and \eqref{b90},
 we arrive at the lower bound 
$$
T_q(\mu)\geq C T_0(\mu)+F
$$
where $F$ is a finite rank operator. 
For the
eigenvalues of $T_0(\mu)$  the required lower
estimates are already obtained by reference to \cite[Lemma~3.1]{FilPush};
 this completes the proof of the lower bound. 
\qed

\section{Proof of Proposition~\ref{th6}}\label{lastsec}

We will prove separately upper and lower bounds. 

\subsection{Proof of the upper bound}

1.
Let $U\subset\R^2$ be an open bounded set with a Lipschitz boundary,
$K\subset U$, and let
$\omega\in C_0^\infty(\R^2)$ be such that $\omega|_K=1$ and
$\omega|_{U^c}=0$.
Denote $\tilde \omega=1-\omega$.
Note that for any $\psi\in\mathcal H$, the function $\tilde\omega P_q\psi$
belongs both to $\Dom(\x_0)$ and to the form domain of $\x(K^c)$. 
Thus, by Proposition~\ref{th4}(ii) (with $A=\x_0$ and $B=\x(K^c)$), we have 
 $W\x_0 \tilde \omega P_q\psi=0$.
Thus, we have
\begin{multline*}
(WP_q\psi,P_q\psi)
=\frac1{\Lambda_q^2}(W\x_0 P_q\psi,\x_0 P_q\psi)
\\
=\frac1{\Lambda_q^2}(W\x_0 (\omega+\tilde \omega)P_q\psi,\x_0 (\omega+\tilde \omega)P_q\psi)
=\frac1{\Lambda_q^2}(W\x_0 \omega P_q\psi,\x_0 \omega P_q\psi).
\end{multline*}
Since $W=\x_0^{-1}-R(K^c)\leq \x_0^{-1}$, we have
$$
(W\x_0\omega P_q\psi,\x_0\omega P_q\psi)
\leq
(\x_0^{-1} \x_0\omega P_q\psi,\x_0\omega P_q\psi)
=
\norm{\omega P_q \psi}_{H_A^1}^2. 
$$
Using \eqref{b13}, we get
\begin{multline*}
\norm{\omega P_q \psi}_{H_A^1}^2
=
\norm{\Q^* \omega P_q \psi}^2 
-
B\norm{\omega P_q \psi}^2 
\leq 
\norm{\Q^* \omega P_q \psi}^2 
\\
=
\norm{\omega \Q^* P_q \psi-2i(\partial \omega)P_q\psi}^2
\leq
2\norm{\Q^* P_q \psi}^2_{L^2(U)}+C_1\norm{P_q\psi}^2_{L^2(U)}.
\end{multline*}

2. Due to the compactness of the embedding $H_0^1(U)\subset L^2(U)$,
for any $\beta>0$ there exists a subspace of $H_0^1(U)$ of a finite codimension
such that for all elements $u$ of this subspace,
$$
\int_{U} \abs{u(x)}^2 dx
\leq
\beta^2 \int_{U} \abs{\nabla u(x)}^2 dx
=
\beta^2 \int_{U} \abs{2 \partial u(x)}^2 dx.
$$
Taking $\beta$ sufficiently small, we obtain
\begin{multline*}
\norm{\Q^* u}_{L^2(U)}
\geq
\norm{2 \partial u}_{L^2(U)}
-
\frac{B}2\norm{\overline{z}u}_{L^2(U)}
\geq
\norm{2\partial u}_{L^2(U)}
-
\frac{B}2\sup_{U}\abs{z}
\norm{u}_{L^2(U)}
\\
\geq
(1-\frac{B}2\beta \sup_{U}\abs{z})\norm{2\partial u}_{L^2(U)}
\geq 
\frac12 \norm{2\partial u}_{L^2(U)}
\geq 
\frac1{2\beta}\norm{u}_{L^2(U)}
\end{multline*}
for all $u$ in our subspace. It follows that on a subspace of $\psi\in L^2(\R^2)$
of a finite codimension,
$$
(WP_q \psi, P_q \psi)_{L^2(\R^2)}
\leq
2 \norm{\Q^* P_q \psi}^2_{L^2(U)}
+
4\beta^2\norm{\Q^*P_q\psi}^2_{L^2(U)}
\leq
C\norm{P_{q+1}\psi}^2_{L^2(U)};
$$
the last inequality holds true by \eqref{1:CrAnn}.

Thus, we have
$$
P_q W P_q
\leq
C_3 P_{q+1} \chi_{U} P_{q+1}+F,
$$
where $\chi_U$ is the characteristic function of $U$,
and $F$ is a finite rank operator.

3. From Proposition~\ref{prop1} we get
$$
\limsup_{n\to\infty} (n!\mu_n)^{1/n}\leq \frac12 B(\CapB U)^2.
$$
Since $U$ can be chosen such that $\CapB U$ is arbitrarily close
to $\CapB K$, we get the required upper bound.

\subsection{Proof of the lower bound}
1.
Let $\gamma$, $J$, $\mu$ be as in section~\ref{restriction}. 
Consider the quadratic form in $L^2(\R^2)$ 
$$
\norm{u}^2_{H_A^1(\R^2)}
+
\int_{\supp\mu}\abs{u(x)}^2d\mu(x)
=
\norm{\x_0^{1/2}u}^2_{L^2(\R^2)}
+
\norm{\gamma J^* \x_0^{1/2}u}^2_{L^2(\R^2)}
$$
defined for $u\in H^1_A(\R^2)$. This form is closed and positively 
defined on $L^2(\R^2)$. 
Denote by $\tilde \x$ the corresponding self-adjoint operator
in $L^2(\R^2)$. We have
$$
\tilde \x
=
{\x_0}+\x_0^{1/2}J\gamma^*\gamma J^* \x_0^{1/2}
=\x_0^{1/2}(I+J\gamma^*\gamma J^*) \x_0^{1/2}
$$
and therefore
$$
\x_0^{-1}-\tilde \x^{-1}
=
\x_0^{-1/2}[J\gamma^*\gamma J^*(I+J\gamma^*\gamma J^*)^{-1}]\x_0^{-1/2}.
$$
Since $\gamma$ is compact by assumption, we have $J\gamma^*\gamma J^*\leq I$ 
on a subspace of a finite codimension. Thus, 
$$
\x_0^{-1}-\tilde \x^{-1}\geq \frac12 {\x_0}^{-1/2}J\gamma^*\gamma J^* \x_0^{-1/2}
$$
on a subspace of finite codimension, and so 
\begin{equation}
P_q(\x_0^{-1}-\tilde \x^{-1})P_q 
\geq
\frac1{2\Lambda_q}(\gamma J^* P_q)^*(\gamma J^* P_q)+F
\label{b1}
\end{equation}
where $F$ is a finite rank operator. 

2.
Now let $K\subset \R^2$ be a compact with a non-empty interior.
Let $K_1\subset K$ be a set with a Lipschitz boundary.
Let $\mu$ be the restriction of the Lebesgue measure on $K_1$.
By Proposition~\ref{th4}(i), we have
$\x_0^{-1}\geq \tilde \x^{-1}\geq R(K^c)$.
It follows that
\begin{equation}
P_q(\x_0^{-1}-R(K^c))P_q
\geq
P_q(\x_0^{-1}-\tilde \x^{-1})P_q.
\label{b2}
\end{equation}
From here, using \eqref{b1} and Proposition~\ref{prop1}(i), we get
the required lower bound in the first part of Proposition~\ref{th6}.
Finally, consider the case of $K$ being a smooth curve.
Let $\mu$ be the arc measure of the curve.
Then, again by \eqref{b1} and \eqref{b2}, and applying Proposition~\ref{prop1}(ii),
we get the second part of  Proposition~\ref{th6}.

 \end{document}